\documentclass[12pt]{article}
\usepackage{amssymb,amsmath,latexsym}
\usepackage{amsthm}
\newtheorem{theorem}{Theorem}

\begin{document}
\author{Alexander E Patkowski}
\title{An interesting $q$-series related to the $4$-th symmetrized rank function}
\date{\vspace{-5ex}}
\maketitle
\abstract{This paper presents the methods to utilizing the $s$-fold extension of Bailey's lemma to obtain $spt$-type functions
related to the symmetrized rank function $\eta_{2k}(n).$ We provide the $k=2$ example, but clearly illustrate
how deep connections between higher-order spt functions exist for any integer $k>1,$ and provide several directions for possible research. In particular,
we present why the function $spt_M^{*}(n),$ the total number of appearances of the smallest parts of partitions where parts greater than the smallest plus $M$ do not occur, is
an $spt$ function that appears to have central importance.
\\*
\\*

\section{Introduction and Main $q$-series}
\par As usual [10], set $(X)_n=(X;q)_{n}:=\prod_{0\le k\le n-1}(1-Xq^{k}).$ In [3], Andrews' contructed an $s$-fold extension of Bailey's lemma,
and obtained many interesting higher dimensional identities. In [5], we find that $spt(n),$ the number of appearances of the smallest parts in the number of partitions of $n,$
satisfies the equation $spt(n)=np(n)-\frac{1}{2}N_2(n).$ Here $N_k(n)=\sum_{m\in\mathbb{Z}}m^kN(m,n),$ where $N(m,n)$ is the number of partitions of $n$ with rank $m$ [1]. 
Here we have kept the usual notation for the number of unrestricted partitions of $n,$ $p(n)$ [1]. \par The results of Garvan [9] extended Andrews' important smallest part identity, allowing for
connections between higher-order spt functions with the symmetrized rank function (defined by equation (7) and [4]) $\eta_{2k}(n),$ for any integer $k>1.$ Other smallest part functions with further restrictions on parts have also been noted in Patkowski [13, 14]. As it turns out, we are able to show there are other spt functions for each $k>1$ that follow from the $k$-fold Bailey lemma. In this sense, the order of the smallest part functions in Garvan's paper [9] correspond to the number of folds in the Bailey lemma in the sense of [3]. This paper will provide the $k=2$ example, which will therefore require only the $2$-fold Bailey lemma.
\\*
\\*
{\bf $2$-fold Bailey's Lemma [3, Theorem 1]} \it We define a pair of sequences $(\alpha_{n_1, n_2},\beta_{n_1, n_2})$ to be a $2$-fold Bailey pair with respect to $a_j,$ $j=1,2,$ if 
		\begin{equation}\beta_{n_1, n_2}=\sum_{r_1\ge0}^{n_1}\sum_{r_2\ge0}^{n_2} \frac{\alpha_{r_1,r_2}}{(a_1q;q)_{n_1+r_1} (q;q)_{n_1-r_1}(a_2q;q)_{n_2+r_2} (q;q)_{n_2-r_2}}.\end{equation}
Furthermore, we have
$$\sum_{n_1\ge0}^{\infty}\sum_{n_2\ge0}^{\infty}(x)_{n_1}(y)_{n_1}(z)_{n_2}(w)_{n_2}(a_1q/xy)^{n_1}(a_2q/zw)^{n_2}\beta_{n_1, n_2}$$
\begin{equation}=\frac{(a_1q/x)_{\infty}(a_1q/y)_{\infty}(a_2q/z)_{\infty}(a_2q/w)_{\infty}}{(a_1q)_{\infty}(a_1q/xy)_{\infty}(a_2q)_{\infty}(a_2q/zw)_{\infty}}\sum_{n_1\ge0}^{\infty}\sum_{n_2\ge0}^{\infty}\frac{(x)_{n_1}(y)_{n_1}(z)_{n_2}(w)_{n_2}(a_1q/xy)^{n_1}(a_2q/zw)^{n_2}\alpha_{n_1, n_2}}{(a_1q/x)_{n_1}(a_1q/y)_{n_1}(a_2q/z)_{n_2}(a_2q/w)_{n_2}}.\end{equation} \rm
In a paper by Joshi and Vyas [11], we find they have studied the special case of the $s$-fold extension. In our interest we will consider their use of the $2$-fold
extension of Bailey's lemma where they have chosen the $\alpha_{n_1, n_2}$ to be $\alpha_{n_1, n_2}=\alpha_n,$ when $n_1=n_2=n,$ and $0$ otherwise. In particular,
we find [11] that $(\alpha_{n_1, n_2},\beta_{n_1, n_2})$ is a $2$-fold Bailey pair with respect to $a_j=1,$ $j=1,2,$ where $\alpha_{0, 0}=1,$

\begin{equation}\alpha_{n_1, n_2}=\begin{cases} (-1)^nq^{n(3n-1)/2}(1+q^n),& \text {if } n_1=n_2=n,\\ 0, & \text{otherwise,} \end{cases}\end{equation}
and 

\begin{equation}\beta_{n_1, n_2}=\frac{1}{(q)_{n_1}(q)_{n_2}(q)_{n_1+n_2}}.\end{equation}
\rm
Differentiating (2) in the same manner as in [13] but appealing to all variables $x,$ $y,$ $z,$ $w,$ when setting equal to $1$ (after setting $a_1=a_2=1$),
\begin{equation}\sum_{n_1\ge1}\sum_{n_2\ge1}(q)_{n_1-1}^2(q)_{n_2-1}^2\beta_{n_1, n_2}q^{n_1+n_2}=\left(\sum_{n\ge1}\frac{nq^n}{1-q^n}\right)^2+\sum_{n_1\ge1}\sum_{n_2\ge1}\frac{\alpha_{n_1, n_2}q^{n_1+n_2}}{(1-q^{n_1})^2(1-q^{n_2})^2}.\end{equation}
Inserting the Joshi and Vyas $2$-fold Bailey pair (3)-(4) into equation (5) and then multiplying through by the generating function for $p(n)$ results in the following $q$-series 
identity.

\begin{theorem} We have,
$$\sum_{n_1\ge1}\left(q^{n_1}+2q^{2n_1}+\cdots\right)\frac{1}{(q^{n_1+1})_{\infty}}\sum_{n_2\ge1}\left(q^{n_2}+2q^{2n_2}+\cdots\right)\frac{1}{(1-q^{n_2+1})\cdots(1-q^{n_2+n_1})}$$
\begin{equation}=\frac{1}{(q)_{\infty}}\left(\sum_{n\ge1}\frac{nq^n}{1-q^n}\right)^2+\frac{1}{(q)_{\infty}}\sum_{\substack{n\in\mathbb{Z}\\ n\neq0}}\frac{(-1)^{n}q^{3n(n+1)/2}}{(1-q^n)^4}.\end{equation}
\end{theorem}
The first sum in Theorem 1 is a rearrangement of 
$$\frac{1}{(q)_{\infty}}\sum_{n_1,n_2\ge1}\frac{(q)_{n_1-1}^2(q)_{n_2-1}^2q^{n_1+n_2}}{(q)_{n_1}(q)_{n_2}(q)_{n_1+n_2}}.$$
Notice that the last $q$-series of Theorem 1 is the $k=2$ instance of the definition
\begin{equation}\sum_{n\ge0}\eta_{2k}(n)q^n=\frac{1}{(q)_{\infty}}\sum_{\substack{n\in\mathbb{Z}\\ n\neq0}}\frac{(-1)^{n+1}q^{n(3n+1)/2+kn}}{(1-q^n)^{2k}},\end{equation}
where $\eta_{k}(n)$ is the $k$-th symmetrized rank [4]
$$\eta_k(n)=\sum_{j=-n}^{n}\binom{j+\lfloor{\frac{k-1}{2}}\rfloor}{k}N(j,n).$$
 In fact [9], $\eta_{4}(n)=(N_4(n)-N_2(n))/24.$
 For the second component of Theorem 1, we use some notation and results from [6]. Define $\delta_q$ to be the differential operator with respect to $q,$ $q\frac{\partial}{\partial q},$ and let $P=P(q)=\frac{1}{(q)_{\infty}}.$
 Then clearly $\delta_q(P)=P\Phi_1(q),$ where
 $$\Phi_i:=\sum_{n\ge1}\frac{n^iq^n}{1-q^n}.$$ Further, from [6, eq.(4.15)] we note
 \begin{equation}\delta_q^2(P)=-\frac{1}{6}P\left(6\Phi_1^2-5\Phi_3-\Phi_1\right),\end{equation}
 and [6, eq.(4.7)]
  \begin{equation}C_4=2P\left(\Phi_3+6\Phi_1^2\right),\end{equation}
  where [6, eq.(1.34)] $C_j:=\sum_{n\ge1}M_j(n)q^n,$ $M_k(n)=\sum_{m\in\mathbb{Z}}m^kM(m,n),$
  and $M(m,n)$ is the number of partitions of $n$ with crank $m$ [2]. Combining (8) and (9) and some simplification we have that 
  \begin{equation}\sum_{n\ge1}n^2p(n)q^n=\frac{5}{12}C_4-6P\Phi_1^2+\frac{1}{6}P\Phi_1,\end{equation}
  or equivalently,
   \begin{equation}P\Phi_1^2=\frac{5}{72}C_4-\frac{1}{6}\sum_{n\ge1}n^2p(n)q^n+\frac{1}{36}P\Phi_1.\end{equation}

Define the sequences $spt_j^{'}(n),$ and $spt_j^{*}(n)$ to be
\begin{equation}\sum_{n\ge1}spt_j^{'}(n)q^n:=\frac{q^{j}}{(1-q^{j})^2(q^{j+1})_{\infty}},\end{equation}
$$\sum_{n\ge1}spt_j^{*}(n)q^n:=\sum_{n_2\ge1}\left(q^{n_2}+2q^{2n_2}+\cdots\right)\frac{1}{(1-q^{n_2+1})\cdots(1-q^{n_2+j})},$$
respectively. Then $\sum_{j\ge1}spt_j^{'}(n)=spt(n),$ and $spt_j^{*}(n)$ is the total number of appearances of the smallest parts of partitions where parts greater than the smallest plus $j$ do not occur. Therefore, 
\begin{equation} \sum_{n\ge1}spt_j^{+}(n)q^n=\frac{q^{j}}{(1-q^{j})^2(q^{j+1})_{\infty}}\sum_{n_2\ge1}\left(q^{n_2}+2q^{2n_2}+\cdots\right)\frac{1}{(1-q^{n_2+1})\cdots(1-q^{n_2+j})},\end{equation}
where $spt_j^{+}(n)=\sum_{k\le n}spt_j^{*}(k)spt_j^{'}(n-k).$ Summing over $j$ we put $SPT^{+}(n):=\sum_{j\ge1}spt_j^{+}(n),$ the coefficient of $q^n$ of the first $q$-series in Theorem 1. Combining the above computations, we may now claim the following corollary, which we state as our next theorem.

\begin{theorem} For each natural number $n,$ we have,
\begin{equation}SPT^{+}(n)=\frac{5}{72}M_4(n)-\frac{1}{6}n^2p(n)+\frac{1}{36}np(n)-\eta_4(n).\end{equation}

\end{theorem}
\section{Some Consequences and Concluding Remarks}
A natural consequence of Theorem 2 and known congruences for $p(n),$ $M_4(n),$ and $\eta_4(n)$ is the following result.

\begin{theorem} We have,
\begin{equation}SPT^{+}(7n)\equiv0\pmod{7},\end{equation}
\begin{equation}SPT^{+}(11n)\equiv0\pmod{11}.\end{equation}
\end{theorem}
\begin{proof} To see (15) note from Garvan [9], that $M_4(7n)\equiv0\pmod{7},$ $spt_2(7n)\equiv0\pmod{7},$ $M_2(7n)\equiv0\pmod{7},$ and therefore $\eta_4(7n)\equiv0\pmod{7},$ and the result follows. To see (16) note from K. Bringmann, F. G. Garvan and K. Mahlburg [8] $\eta_4(11n)\equiv0\pmod{11},$ and Garvan [9] that $M_4(11n)\equiv0\pmod{11}.$ Collecting these properties gives the result.
\end{proof}

If we were to delve deeper into $SPT^{+}(n)$ and use the ``crank component" akin to the study done in [14], we could easily construct a relation to Garvan's second order function (see [9] for a definition) $spt_2(n).$ This is easily 
done using his relation $spt_2(n)=\mu_4(n)-\eta_4(n).$
\par A $1$-fold Bailey pair (i.e. Bailey pair) $(\alpha_n, \beta_n)$ with $\beta_n=1/((q)_n(q)_{n+m}),$ for $m\in\mathbb{N}$ has been obtained in [15]. The $q$-series therein is \begin{equation}\sum_{n_2\ge1}\left(q^{n_2}+2q^{2n_2}+\cdots\right)\frac{1}{(1-q^{n_2+1})\cdots(1-q^{n_2+j})},\end{equation}
and is a natural generalization of Andrews' $spt(n).$ This identity is follows from the $1$-fold Bailey, and has come about in our study is the same sense as in [12]. A clear and interesting area for further research that has come about from our study, is to extend the results of Andrews [5] to obtain the divisibility properties of $spt_M^{*}(n).$ This is quite challenging, and perhaps the ``easiest" angle is through purely combinatorial methods. Another area of needed research is to further solidify the connection of the dimension characterized in [3] and the order characterized in [9], which was first started in this study.

1390 Bumps River Rd. \\*
Centerville, MA
02632 \\*
USA \\*
E-mail: alexpatk@hotmail.com
\end{document}